\documentclass[12pt]{article}

\title{Free resolutions of algebras}
\author{Joe Chuang \and Alastair King}
\date{18 Oct 2012}

\usepackage{tikz}
\usepackage{url}

\input xypic
\usepackage[all]{xy}

\usepackage{amssymb,amsthm,array}
\usepackage[reqno]{amsmath}

\theoremstyle{plain}
\newtheorem{theorem}{Theorem}[section]
\newtheorem{proposition}[theorem]{Proposition}
\newtheorem{lemma}[theorem]{Lemma}
\newtheorem{corollary}[theorem]{Corollary}

\theoremstyle{definition}

\newtheorem{example}[theorem]{Example}
\newtheorem{remark}[theorem]{Remark}

\numberwithin{equation}{section}

\newcommand{\ZZ}{\mathbb{Z}}
\newcommand{\AM}{\mathcal{A}}
\newcommand{\AF}{\mathcal{F}}
\newcommand{\KF}{K}
\newcommand{\isom}{\cong}
\newcommand{\subs}{\subseteq}

\newcommand{\tensalg}[2]{T^+_{#1}\left(#2\right)}
\newcommand{\gr}[1]{^{(#1)}}
\newcommand{\qiso}{\simeq}
\newcommand{\grd}{_\bullet}

\newcommand{\comult}{\Delta}
\newcommand{\cobar}{\operatorname{Cobar}}
\newcommand{\Hom}{\operatorname{Hom}}
\newcommand{\Ext}{\operatorname{Ext}}
\newcommand{\Tor}{\operatorname{Tor}}
\newcommand{\id}{\operatorname{id}}

\newcommand{\ass}{\mathcal{S}}
\newcommand{\B}{B}
\newcommand{\Bi}{B^\infty}

\newcommand{\bra}{\texttt{[}}
\newcommand{\ket}{\texttt{]}}
\newcommand{\inpt}{\bullet}
\newcommand{\Brakc}[1]{\mathcal{B}^{[#1]}}
\newcommand{\Brako}[1]{\mathcal{B}^{(#1)}}

\newcommand{\AG}{\operatorname{Gr}}

\begin{document}
\maketitle
\begin{abstract}
Given an algebra $A$, presented by generators and relations,
i.e. as a quotient of a tensor algebra by an ideal,
we construct a free algebra resolution of $A$,
i.e. a differential graded algebra which is quasi-isomorphic to
$A$ and which is itself a tensor algebra.
The construction rests combinatorially on the set of bracketings
that arise naturally in the description of a free contractible differential
graded algebra with given generators.
\end{abstract}

\goodbreak\section{Introduction} \label{sec1}

Let $S$ be any ring.
We write $\otimes$ for $\otimes_S$, 
and, for any $S,S$-bimodule $V$, we write $V^n$ for $V^{\otimes n}$.
Further, we write
\begin{equation}
 \tensalg{S}{V} = \bigoplus_{n\geq 1} V^n,
 \qquad
  V\gr{m} = \bigoplus_{n\geq m} V^n,
\end{equation}
for the non-unital tensor algebra of `free words' in $V$ 
over $S$ and for the ideal of words of length at least $m$, respectively.

Our objective in this paper is to construct a free algebra resolution
of an arbitrary (non-unital) $S$-algebra
$A$ presented in terms of generators and relations, that is, 
\[
  A=\tensalg{S}{V}/I,
\] 
for some $S,S$-bimodule $V$ and ideal $I$.
As a first step, in Section~\ref{sec2}, we work towards describing the case
where $I\subs V\gr{2}$, i.e. the generators are `minimal', and $I$ is homogeneous
(Theorem~\ref{thm:homog}).
This case is controlled by the by-now-familiar $A_\infty$ combinatorics
of rooted trees or the corresponding set of bracketings 
(see, for example, Figure~\ref{fig:bra4}).
By extending this set of bracketings in various ways, 
we find, in Sections~\ref{sec3} and \ref{sec4},
that the essential nature of the proof becomes more
transparent in the general case (Theorem~\ref{thm:main}).

As a typical motivating example, consider a quiver $Q$ and field $k$.
Set $S=k^{Q_0}$, the semi-simple algebra spanned by the vertex idempotents, 
and $V=k^{Q_1}$, the $S,S$-bimodule spanned by the arrows. 
Then the augmented unital algebra $S\oplus\tensalg{S}{V}$ 
is the path algebra $kQ$ and $S\oplus A$ is a `quiver algebra', 
i.e. is presented by a quiver with relations.
Such augmentation allows one to move easily between unital 
and non-unital $S$-algebras;
we find it notationally simpler to work in the
slightly less familiar non-unital context in this paper.

Giving a free resolution of $A$ as an $S$-algebra is,
almost tautologically, the same as giving 
a `locally finite' $A_\infty$-coalgebra over $S$
\[
  K\grd =\bigoplus_{n\geq 1} K_n,
\]
where each $K_n$ is an $S,S$-bimodule
and such that $A$ is quasi-isomorphic to 
the differential graded (dg) algebra
\begin{equation}
 \cobar{ K\grd } = \bigl( \tensalg{S}{K\grd[1]},d \bigr).
\end{equation}
Here $K\grd[1]$ denotes the shifted complex, defined by $K[1]_n=K_{n+1}$.
In general, the quasi-isomorphism should be induced by some $S,S$-morphism 
\[
  \theta\colon K_1\to A
\]
such that $\theta(K_1)$ generates $A$.
In our case, since the generators $V$ of $A$ are explicitly specified,
we will suppose that  
$K_1=V$ and that $\theta\colon V\to A$ is the specifying map. 
Then the degree zero term in $\cobar{ K\grd }$ is $\tensalg{S}{V}$
and the quotient map $\tensalg{S}{V} \to A$ will be the quasi-isomorphism.

The cobar construction here is `almost tautological' in that
the differential $d$ on $\tensalg{S}{K\grd[1]}$ 
may be considered simply as an efficient way of encoding the $A_\infty$-coalgebra 
structure on $K\grd$.
More precisely, the $A_\infty$-coalgebra structure consists of
$S,S$-morphisms 
\[
  \Delta_k\colon K_n \to K_{m_1}\otimes \cdots \otimes K_{m_k},
\]
for each $k\geq 1$ and $\sum_{i=1}^k m_i=n-2+k$,
satisfying the conditions
\[
\sum_{r+s+t=n} (-1)^{r+st} 
 \bigl(1^{\otimes r} \otimes \Delta_s \otimes 1^{\otimes t}\bigr)
 \circ \Delta_{r+1+t} =0 
\]
for all $n\geq 1$ (cf. \cite[Definition 1.2.1.8]{Lef}).
These may be packaged into the single condition
that the endomorphism $d$ of $\tensalg{S}{K\grd[1]}$, determined by 
\begin{equation}
\label{eq:d}
d=-[1]^{\otimes n}\circ \Delta_n\circ [-1] : K\grd[1] \to K\grd[1]^{\otimes n}
\end{equation}
and the (graded) Leibniz rule (see \eqref{eq:leftLeib}),
satisfies $d^2=0$ (cf. \cite[\S 1.2]{Lef}).
Note also that, in evaluating $[1]^{\otimes n}$,
we use the standard Koszul sign rule; 
for example 
\[
 [1]\otimes[1]: = (-1)^{m}: K_m\otimes K_n \to K_m\otimes K_n.
\]

\begin{remark}
The condition that $K\grd$ is `locally finite' is simply the requirement
that \eqref{eq:d} does define a differential on $\tensalg{S}{K\grd[1]}$,
which is defined as an infinite direct sum, and not just on its completion,
i.e. the corresponding direct product. In other words, 
for any element $x\in K\grd$, the coproducts $\Delta_n(x)$ are non-zero
for only finitely many $n$. It will be clear that all coalgebras 
we will encounter have this property 
and we will not explicitly mention it again.
\end{remark}

\goodbreak\section{First examples} \label{sec2}

We begin by discussing the sort of construction we are looking for
in the case that $I=V\gr{2}$ and
so $A=\tensalg{S}{V}/I$ is just $V$ with trivial multiplication. 

As a warm-up, we first observe that
the classical candidate for $K\grd$, in this case, is
the free coassociative coalgebra generated by $V$,
\begin{equation}
  \B\grd(V) = \bigoplus_{n\geq 1} V^n,
\end{equation}
which is what the usual (unaugmented) bar construction yields.
The comultiplication 
$\comult\colon \B_{m+n}\to \B_m\otimes \B_n$
is tautological, 
i.e. is the natural identification 
$\tau\colon V^{m+n} \to V^{m}\otimes V^{n}$.
This gives the $A_\infty$-coproduct $\comult_2$,
with all other coproducts vanishing.

As one should expect, the dg algebra $\AM\grd = \cobar\B\grd(V)$
is a free resolution of $A$. 
Indeed, we may explicitly write
\begin{equation}
 \AM\grd = \bigoplus_{\pi\in \Pi} V^\pi,
\end{equation}
where $\Pi$ is the set of all finite sequences $\pi=(\pi_1,\dots,\pi_r)$
of positive integers, while 
\[ V^{\pi}=V^n, \quad\text{for $n=\sum_k \pi_k$,}
\]
which is in homological degree $d$, i.e. is a summand of $\AM_d$, for   
\[
  d=|\pi|=\sum_k (\pi_k-1).
\]  
Multiplication in $\AM\grd$ corresponds to concatenation of sequences,
i.e. is given by the tautological maps 
$\tau\colon V^\pi\otimes V^\eta\to V^{\pi\eta}$,
while the non-trivial components of the differential 
$d\colon V^\pi\to V^{\pi'}$ are $\pm \tau$  
whenever $\pi'$ is obtained from 
$\pi$ by splitting some term in the sequence into two.
The sign is $(-1)^{\left|\left(\pi'_1,\ldots,\pi'_p\right)\right|}$
when the $p$-th term of $\pi$ is split.

To see explicitly that $\AM\grd$ is quasi-isomorphic to the algebra $A$,
we notice first that, for each $n\geq 0$, there is precisely one sequence 
of degree $0$ summing to $n$, namely $\pi=(1,\dots,1)$, 
and this yields $\AM_0=\tensalg{S}{V}$.
On the other hand, for each $n\geq 2$, the part of $\AM\grd$ consisting
of summands equal to $V^n$ is given by 
$V^n\otimes_\ZZ C\grd^{aug}(\sigma_{n-2})[1]$
the (shifted and augmented) chain complex of the $(n-2)$-simplex, 
which is exact, as required.

From the point-of-view of this paper, a more natural, but rather bigger,
candidate for $K\grd$ is the free $A_\infty$-coalgebra generated by $V$,
\begin{equation}
\label{eq:Kfree}
  \Bi\grd(V) = \bigoplus_{\beta\in\Brakc2} V^\beta .
\end{equation}
Here $\Brakc2$ is the set of all closed non-degenerate bracketings
(or equivalently those that correspond to `rooted trees') and the
bracketed tensor product $V^\beta=V^n$
if $\beta$ has $n$ inputs and is in homological degree $d$ 
if $\beta$ has $d-1$ pairs of brackets.
Note that ``closed'' means that the whole expression is 
enclosed in an outer bracket,
while ``non-degenerate'' means that each inner pair of 
brackets encloses at least two inputs.
For example, one summand of $\Bi_4(V)$ would be 
\[
  V^\beta = \bra\bra V\otimes V \ket
           \otimes \bra V\otimes V\otimes V \ket \ket = V^5
\]
for $\beta=\bra\bra\inpt\inpt\ket\bra\inpt\inpt\inpt\ket\ket$
or $\bra\bra2\ket\bra3\ket\ket$, a 3-fold bracketing of 5 inputs.

The differential $\Delta_1$ has non-zero components 
\[
  (-1)^{m-1}\tau\colon V^\beta \to V^{\beta(\widehat{m})},
\]
where $\beta(\widehat{m})$ is obtained from $\beta$ by removing the $m$th internal left bracket $\bra$ (counted from the left) together with its matching right bracket $\ket$.
Precisely one higher coproduct $\Delta_k$, for some $k\geq 2$, is defined on each $V^\beta$ 
and corresponds to removing the outer bracket and writing what is inside as a concatenation
of $k$ closed bracketings, up to a sign. For example, there is a component 
\[
  \Delta_4 \colon V^{\bra\bra2\ket2\bra3\ket\ket} 
  \to V^{\bra2\ket}\otimes V\otimes V\otimes V^{\bra3\ket}.
\]
For the general component
\[
  \Delta_k \colon V^{\bra\beta_1\ldots\beta_k\ket} 
  \to V^{\beta_1}\otimes\ldots\otimes V^{\beta_k}
\]
the sign we choose (see Remark~\ref{rem:AFcoalg} for an explanation)
is
\begin{equation}
\label{eq:sign!}
    -(-1)^{\sum_{i=1}^{k}(k-i)d_i},
\end{equation}
where $d_i$ is the homological degree of $V^{\beta_i}$ in $\Bi\grd(V)$.

By convention, the only 0-fold closed bracketing is 1 and so $\Bi_1(V)=V$.
On the other hand, the 1-fold closed bracketings 
are $\bra k\ket$, for each $k\geq 2$, and hence
$\Bi_2(V) = V\gr{2}$.
As a further example, 
the possible bracketings of 4 inputs are 3, 2 or 1-fold, 
as listed in Figure~\ref{fig:bra4}
together with their corresponding rooted trees.
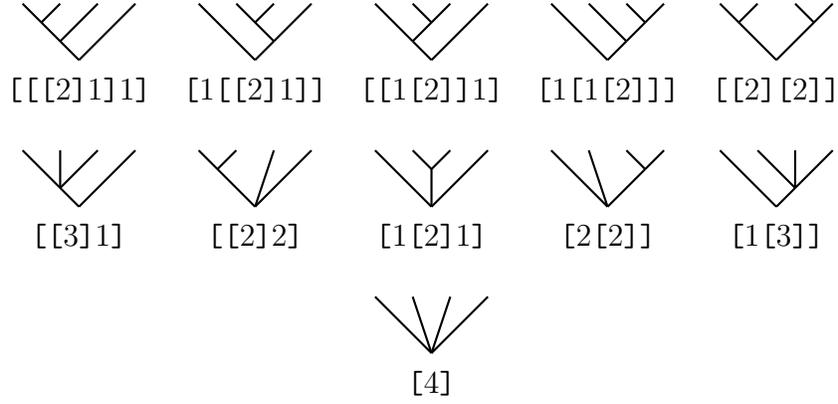
\begin{figure}
\begin{center}
\begin{tabular}{ccccc}
\begin{tikzpicture}[scale=0.5]
\foreach \j in {1,2,3,4} {\coordinate (a\j) at (\j,1.5);}
\foreach \j in {1,2,3} {\coordinate (b\j) at (0.5+\j,1.0);}
\foreach \j in {1,2} {\coordinate (c\j) at (1.0+\j,0.5);}
\coordinate (z) at (2.5,0);
\foreach \t/\h in {a1/b1, a2/b1, a3/c1, a4/z, b1/c1, c1/z} 
  {\draw[thick] (\t)--(\h);} 
\end{tikzpicture} 
&
\begin{tikzpicture}[scale=0.5]
\foreach \j in {1,2,3,4} {\coordinate (a\j) at (\j,1.5);}
\foreach \j in {1,2,3} {\coordinate (b\j) at (0.5+\j,1.0);}
\foreach \j in {1,2} {\coordinate (c\j) at (1.0+\j,0.5);}
\coordinate (z) at (2.5,0);
\foreach \t/\h in {a1/z, a2/b2, a3/b2, a4/c2, b2/c2, c2/z}
  {\draw[thick] (\t)--(\h);} 
\end{tikzpicture} 
&
\begin{tikzpicture}[scale=0.5]
\foreach \j in {1,2,3,4} {\coordinate (a\j) at (\j,1.5);}
\foreach \j in {1,2,3} {\coordinate (b\j) at (0.5+\j,1.0);}
\foreach \j in {1,2} {\coordinate (c\j) at (1.0+\j,0.5);}
\coordinate (z) at (2.5,0);
\foreach \t/\h in {a1/c1, a2/b2, a3/b2, a4/z, b2/c1, c1/z}
  {\draw[thick] (\t)--(\h);} 
\end{tikzpicture} 
&
\begin{tikzpicture}[scale=0.5]
\foreach \j in {1,2,3,4} {\coordinate (a\j) at (\j,1.5);}
\foreach \j in {1,2,3} {\coordinate (b\j) at (0.5+\j,1.0);}
\foreach \j in {1,2} {\coordinate (c\j) at (1.0+\j,0.5);}
\coordinate (z) at (2.5,0);
\foreach \t/\h in {a1/z, a2/c2, a3/b3, a4/b3, b3/c2, c2/z}
  {\draw[thick] (\t)--(\h);} 
\end{tikzpicture} 
&
\begin{tikzpicture}[scale=0.5]
\foreach \j in {1,2,3,4} {\coordinate (a\j) at (\j,1.5);}
\foreach \j in {1,2,3} {\coordinate (b\j) at (0.5+\j,1.0);}
\foreach \j in {1,2} {\coordinate (c\j) at (1.0+\j,0.5);}
\coordinate (z) at (2.5,0);
\foreach \t/\h in {a1/b1, a2/b1, a3/b3, a4/b3, b1/z, b3/z}
  {\draw[thick] (\t)--(\h);} 
\end{tikzpicture} 
\\
  $\bra\bra\bra2\ket1\ket1\ket$ 
& $\bra1\bra\bra2\ket1\ket\ket$ 
& $\bra\bra1\bra2\ket\ket1\ket$ 
& $\bra1\bra1\bra2\ket\ket\ket$ 
& $\bra\bra2\ket\bra2\ket\ket$ \\
&&&&\\
\begin{tikzpicture}[scale=0.5]
\foreach \j in {1,2,3,4} {\coordinate (a\j) at (\j,1.5);}
\foreach \j in {1,2} {\coordinate (c\j) at (1.0+\j,0.5);}
\coordinate (z) at (2.5,0);
\foreach \t/\h in {a1/c1, a2/c1, a3/c1, a4/z, c1/z}
  {\draw[thick] (\t)--(\h);} 
\end{tikzpicture} 
&
\begin{tikzpicture}[scale=0.5]
\foreach \j in {1,2,3,4} {\coordinate (a\j) at (\j,1.5);}
\foreach \j in {1,2,3} {\coordinate (b\j) at (0.5+\j,1.0);}
\foreach \j in {1,2} {\coordinate (c\j) at (1.0+\j,0.5);}
\coordinate (z) at (2.5,0);
\foreach \t/\h in {a1/b1, a2/b1, a3/z, a4/z, b1/z}
  {\draw[thick] (\t)--(\h);} 
\end{tikzpicture} 
&
\begin{tikzpicture}[scale=0.5]
\foreach \j in {1,2,3,4} {\coordinate (a\j) at (\j,1.5);}
\foreach \j in {1,2,3} {\coordinate (b\j) at (0.5+\j,1.0);}
\foreach \j in {1,2} {\coordinate (c\j) at (1.0+\j,0.5);}
\coordinate (z) at (2.5,0);
\foreach \t/\h in {a1/z, a2/b2, a3/b2, a4/z, b2/z}
  {\draw[thick] (\t)--(\h);} 
\end{tikzpicture} 
&
\begin{tikzpicture}[scale=0.5]
\foreach \j in {1,2,3,4} {\coordinate (a\j) at (\j,1.5);}
\foreach \j in {1,2,3} {\coordinate (b\j) at (0.5+\j,1.0);}
\foreach \j in {1,2} {\coordinate (c\j) at (1.0+\j,0.5);}
\coordinate (z) at (2.5,0);
\foreach \t/\h in {a1/z, a2/z, a3/b3, a4/b3, b3/z}
  {\draw[thick] (\t)--(\h);} 
\end{tikzpicture} 
&
\begin{tikzpicture}[scale=0.5]
\foreach \j in {1,2,3,4} {\coordinate (a\j) at (\j,1.5);}
\foreach \j in {1,2,3} {\coordinate (b\j) at (0.5+\j,1.0);}
\foreach \j in {1,2} {\coordinate (c\j) at (1.0+\j,0.5);}
\coordinate (z) at (2.5,0);
\foreach \t/\h in {a1/z, a2/c2, a3/c2, a4/c2, c2/z}
  {\draw[thick] (\t)--(\h);} 
\end{tikzpicture} 
\\
  $\bra\bra3\ket1\ket$ 
& $\bra\bra2\ket2\ket$ 
& $\bra1\bra2\ket1\ket$ 
& $\bra2\bra2\ket\ket$ 
& $\bra1\bra3\ket\ket$ \\
&&&&\\
&& 
\begin{tikzpicture}[scale=0.5]
\foreach \j in {1,2,3,4} {\coordinate (a\j) at (\j,1.5);}
\foreach \j in {1,2,3} {\coordinate (b\j) at (0.5+\j,1.0);}
\foreach \j in {1,2} {\coordinate (c\j) at (1.0+\j,0.5);}
\coordinate (z) at (2.5,0);
\foreach \j in {1,2,3,4} {\draw[thick] (a\j)--(z);} 
\end{tikzpicture} 
&& \\
&& $\bra4\ket$ &&
\end{tabular}
\end{center}
\caption{Bracketings with four inputs.}
\label{fig:bra4}
\end{figure}
These are well-known to correspond to the 0,1 and 2 dimensional 
cells of a pentagon.
Furthermore, the restriction of the differential on $\Bi\grd$ 
corresponds to the coboundary map on the cochain complex of the pentagon
(with some care needed over signs).

More generally, writing $\Bi_{n,k}(V)$
for the sum over all $(n-1)$-fold bracketings of $k$ inputs,
we have for $k\geq 2$,
\begin{equation}
\label{eq:ass} 
   \Bi_{n,k}(V) \isom V^k \otimes_\ZZ C^{k-n}(\ass_{k-2}),
\end{equation}
where $C^\bullet(\ass_{m})$ is the cochain complex of the 
$m$th associahedron (or Stasheff polytope, introduced in \cite{St63}).
In addition, the differential on $C^\bullet(\ass_{k-2})$ induces the differential on
$\Bi_{\bullet,k}(V)$.
Since the associahedra are all contractible, this means that 
the homology of $\Bi_{\bullet,k}(V)$ is just $V^k$ in degree $k$.
In other words, the kernel of $d$ restricted to $\Bi_{k,k}(V)$
is isomorphic to $\B_k(V)$ and the induced map
\begin{equation}
\label{eq:eta}
  \eta\colon \B\grd(V)\to \Bi\grd(V),
\end{equation}
is a quasi-isomorphism of $A_\infty$ coalgebras
(essentially because the $A_\infty$ operad resolves the associative operad).

Using Corollary~\ref{cor:qi=we} from the Appendix,
this is actually sufficient to prove the following result, 
but we will also prove it directly as a special case of 
our main result Theorem~\ref{thm:main} in Section~\ref{sec4}.
 
\begin{theorem}
\label{thm:V}
$\cobar{\Bi\grd(V)}\qiso V$.
\end{theorem}

More generally, suppose $A$ is a graded algebra, generated by $V$ in degree 1, 
so that $A=\tensalg{S}{V}/I$, where 
\[
  I=\bigoplus_{n\geq 2} R_n
\]
is a homogeneous ideal in $V^{(2)}$.
Then, generalising \eqref{eq:Kfree}, we define
\begin{equation}
\label{eq:BiVI}
  \Bi\grd(V,I) = \bigoplus_{\beta\in\Brakc2} (V,I)^\beta,
\end{equation}
where $(V,I)^\beta$ is obtained from $V^\beta$ 
by replacing every occurrence of an inner bracketed 
$\bra V^n \ket$ by $\bra R_n \ket$.
The $A_\infty$ coalgebra structure on $\Bi\grd(V,I)$ is defined, as for $\Bi\grd(V)$,
by tautological maps (with the appropriate sign), 
except for the component of the differential
corresponding to the removal of an inner bracket.
Such a component is induced by one of the two maps
\begin{gather*}
   V^a\otimes \bra R_n\ket \otimes V^b \to R_{n+a+b} , \\
  V^a\otimes \bra R_n\ket \otimes V^b \to V^{n+a+b} ,
\end{gather*}
depending on whether the domain is enclosed by matching brackets $\bra .. \ket$
or non-matching ones, e.g. $\ket .. \ket$.
The first map exists because $I$ is an ideal, while the second
is the composite of the first with the inclusion $R_{n+a+b}\subs V^{n+a+b}$.
This inclusion also gives the degree 2 component 
$\Delta_k\colon \bra R_k \ket \to V^k$.

Another special case of Theorem~\ref{thm:main} is then the following.

\begin{theorem}
\label{thm:homog}
$\cobar{\Bi\grd(V,I)}\qiso A$.
\end{theorem}

\begin{example}
\label{ex:Koszul}
We look more closely at the case when $S$ is a separable $k$-algebra,
so that $\otimes$ is exact and hence $\Bi\grd(V,I)\subs \Bi\grd(V)$.
Now Theorem~\ref{thm:homog} implies 
(by Remark~\ref{rem:main3} and Proposition~\ref{pr:KFtor}, 
or directly from Proposition~\ref{prop:Tor})
that the degree $k$ part of $\Tor^A_n(S,S)$
is isomorphic to the homology of $d$ at $\Bi_{n,k}(V,I)$,
so this degree $k$ part vanishes for $k<n$, 
since there are no bracketings in this case.
Since $\Ext_A^n(S,S)\isom \Hom_S(\Tor^A_n(S,S),S)$,
this also starts in degree $n$ (cf. \cite[Lemma 2.1.2]{BGS}).

Observe further that all inner brackets in 
the bracketings that contribute to $\Bi_{n,n}(V)$ are $\bra2\ket$
(as for example in the top row of Figure~\ref{fig:bra4}).
Hence, under the embedding $\eta\colon \B\grd(V)\to \Bi\grd(V)$
of \eqref{eq:eta},
the kernel of $d$ restricted to $\Bi_{n,n}(V,I)$ is identified
with the familiar Koszul term $\B_n(V,R_2)\subs \B_n(V)$
(cf. \cite[\S 2.6]{BGS}), defined by
\[
 \B_n(V,R_2) = \bigcap_{p+q=n-2} V^p\otimes R_2\otimes V^q. 
\]
Thus, as is also well-known from Koszul theory 
(cf. \cite[Thm 2.6.1]{BGS}),
the degree $n$ part of $\Tor_n(S,S)$
is isomorphic to $\B_n(V,R_2)$.
Note also that $\B\grd(V,R_2)$ is a subcoalgebra of $\B\grd(V)$.

But now, if $\Ext^n_A(S,S)$, and hence $\Tor_n^A(S,S)$, is
concentrated in degree $n$, 
which is one characterisation of $A$ being Koszul
(cf. \cite[Prop 2.1.3]{BGS}), 
then necessarily $I=(R_2)$ and  
we also deduce that $\B\grd(V,R_2)\qiso\Bi\grd(V,I)$ 
and hence, by Corollary~\ref{cor:qi=we},
that $\cobar{\B\grd(V,R_2)}\qiso A$,
which is the algebra incarnation of the Koszul resolution for $A$.
\end{example}

\goodbreak\section{A free contractible dg algebra} \label{sec3}

We now develop the machinery that will enable us to generalise the
constructions behind Theorems~\ref{thm:V} and \ref{thm:homog}.
In the process, we are naturally led to consider bracketings of a slightly more
general form than in Section~\ref{sec2}.

For any ring $S$ and any $S,S$-bimodule $V$, we can construct 
a free contractible dg $S$-algebra
\[
  \AF\grd = \AF\grd(V) = \bigoplus_{n\geq0} \AF_n(V)
\]
that is `freely' generated by $V$ in degree 0 and 
a contracting homotopy $h\colon\AF\grd\to\AF\grd$ of degree 
$1$, satisfying $dh+hd=\id$ and $h^2=0$.
The differential $d$, of degree $-1$, 
is then determined recursively by the two conditions
\begin{eqnarray}
\label{eq:conhom}
  d(hv) &=& v - h(dv), \\
\label{eq:leftLeib}
  d(v\cdot w) &=& dv \cdot w + (-1)^{\deg v} v\cdot dw,
\end{eqnarray}
starting from $dv=0$ for $v\in V$.
Notice that (with these signs) we deduce inductively that $d^2=0$,
by observing that
\[
   d^2hv = dv - dhdv = dv - (dv - hd^2v) = h d^2v
\]
and
\begin{gather*}
  d^2(v\cdot w) = d\bigl(dv \cdot w + (-1)^{\deg v} v\cdot dw\bigr) \\
   = d^2v \cdot w + (-1)^{\deg dv} dv\cdot dw 
   + (-1)^{\deg v} dv\cdot dw + v\cdot d^2 w \\
   = d^2v \cdot w + v\cdot d^2 w,
\end{gather*}
since $\deg(dv)=(\deg v) - 1$.

We may describe $\AF\grd$ more explicitly
using an extended notion of bracketed tensor products, 
similar to Section~\ref{sec2},
\begin{equation}
\label{eq:AF}
 \AF\grd(V) = \bigoplus_{\beta\in\Brako1} V^\beta,
\end{equation}
where, if $\beta$ is a $d$-fold
bracketing of $n$ inputs,
then $V^\beta$ is $V^n$ and is in homological degree $d$,
i.e. is a summand of $\AF_d$.
Note that this degree is different from \eqref{eq:Kfree},
because we are not describing an $A_\infty$-coalgebra here,
but its cobar construction directly.

We define the set $\Brako1$ of bracketings recursively as follows,
noting that each such bracketing is, in the first instance, a word in 
the symbols ``$\inpt$'' (representing an input), ``\bra'' and ``\ket'', 
with the later two balanced in the usual way of brackets:
\begin{enumerate}
\item[(i)] $\Brako1$ contains $\inpt$,
\item[(ii)] if $\beta\in\Brako1$ and $\beta\neq \bra\alpha\ket$, for some 
$\alpha\in\Brako1$, then $\bra\beta\ket\in\Brako1$,
\item[(iii)] if $\alpha,\beta\in\Brako1$, then their concatenation 
$\alpha\beta\in\Brako1$
\end{enumerate}
Note that, in contrast to $\Brakc2$ in \eqref{eq:Kfree}, bracketings in $\Brako1$
may be open, i.e. without an outer bracket, and degenerate, 
i.e. with only one input in an inner bracket.

As before, we abbreviate a sequence of $n$ uninterrupted $\inpt$'s 
by ``$n$''.
Thus, for example, $\bra\inpt\inpt\ket\bra\inpt\inpt\bra\inpt\ket\ket$ becomes
$\bra2\ket\bra2\bra1\ket\ket$.
In particular, there are just two bracketings  $1$ and $\bra1\ket$ with one input
and eight with two inputs
\[
2,\,\bra2\ket
,\,
\bra1\ket 1,\,
1\bra1\ket,\,
\bra\bra1\ket 1\ket,\,
\bra1\bra1\ket\ket,\,
\bra1\ket \bra1\ket,\,
\bra\bra1\ket \bra1\ket\ket.
\]

Now the product in $\AF\grd$ has non-zero components consisting of the 
tautological maps
\[
  \tau\colon V^\alpha\otimes V^\beta \to V^{\alpha\beta},
\]
while the contracting homotopy $h\colon\AF\grd\to\AF\grd$ has non-zero 
components given by the tautological maps
\begin{equation}
\label{eq:Fhom}
  \tau\colon V^\beta \to V^{\bra\beta\ket},
 \quad\text{for each $\beta\neq\bra\alpha\ket$,}
\end{equation}
that is, $\tau$ is the identity map $V^k\to V^k$,
where $k$ is the number of inputs in $\beta$ (and in $\bra\beta\ket$).

On the other hand, the differential $d\colon\AF_n\to\AF_{n-1}$
has $n$ non-zero components
\begin{equation}
\label{eq:Fdiff}
  (-1)^{m-1}\tau\colon V^\beta \to V^{\beta(\widehat{m})},
 \quad\text{for $m=1,\dots,n$,}
\end{equation}
where $\beta(\widehat{m})$ is obtained from $\beta$ by removing the 
$m$th ``$\bra$'' from the left, together with its matching ``$\ket$''.
It is straightforward to check that this does give a differential satisfying
\eqref{eq:conhom} and \eqref{eq:leftLeib},
noting that we count ``$\bra$''s from the 
left, because \eqref{eq:leftLeib} is a `left' Leibniz rule.
Thus $h$ is a contracting homotopy and 
so $\AF\grd$ is a contractible dg algebra.

\begin{remark}
\label{rem:AFcoalg}
Forgetting the differential $d$ and the contracting homotopy $h$,
$\AF\grd$ is the free graded $S$-algebra generated by 
\begin{equation}
\label{eq:coalgKV}
 \KF\grd(V)[1] = V\oplus h(\AF\grd) = \bigoplus_{\beta\in\Brakc1} V^\beta,
\end{equation}
where $\Brakc1\subs\Brako1$ is the set of closed bracketings.
Note that the convention that ``1'' is a closed bracketing is precisely
to get the initial summand $V$ here.

In other words, $\AF\grd=\cobar \KF\grd(V)$.
The differential and coproducts on the $A_\infty$-coalgebra 
$\KF\grd(V)$ are determined by applying \eqref{eq:d}
in reverse and one can check that they are given explicitly
by the same rules as those of $\Bi\grd(V)$
in the paragraphs following \eqref{eq:Kfree}.
In particular, this explains the choice of sign in \eqref{eq:sign!}.
\end{remark}

\goodbreak\section{Main theorem} \label{sec4}

\newcommand{\str}{*}

Note that $\AF_0(V)=\tensalg{S}{V}$ and, furthermore, that we have 
many other copies of $\tensalg{S}{V}$ `embedded' in $\AF\grd$ 
for every pair of inner brackets. 
To make this explicit, we may introduce a new symbol $\bra\str\ket$ with the meaning
\[
  V^{\bra\str\ket} = \bigoplus_{m\geq1} V^{\bra m\ket} = \bra \tensalg{S}{V} \ket
\]
and more generally
\[
  V^{\alpha\bra \str\ket\beta}  = \bigoplus_{m\geq1} V^{\alpha\bra m\ket\beta} 
  = V^{\alpha} \otimes \bra \tensalg{S}{V} \ket \otimes V^{\beta},
\]
for matching partial bracketings $\alpha,\beta$.
Thus we can define contracted sets $\Brakc\str$ of closed bracketings and $\Brako\str$ of open bracketings,
in which the innermost brackets are all  $\bra\str\ket$, 
and so that we can write
\begin{equation}
\label{eq:AFstr}
 \AF\grd(V) = \bigoplus_{\beta\in\Brako\str} V^\beta,
 \qquad
 \KF\grd(V) = \bigoplus_{\beta\in\Brakc\str} V^\beta,
\end{equation}
where the degrees of terms in the 2nd equation
are shifted compared to \eqref{eq:coalgKV}.
Now, for any ideal $I\subs\tensalg{S}{V}$ we can define a new
dg algebra 
\begin{equation}
\label{eq:FVI}
\AF\grd(V,I)= \bigoplus_{\beta\in\Brako\str} (V,I)^\beta,
\end{equation}
where $(V,I)^\beta$ is obtained from $V^\beta$ by replacing 
$\bra \tensalg{S}{V}\ket$ by $\bra I \ket$, for each occurrence of $\bra\str\ket$ in $\beta$.
Because $I$ is an ideal, there is a well-defined differential 
$d\colon \AF\grd(V,I)\to\AF\grd(V,I)$
given by the same rule \eqref{eq:Fdiff} as the differential 
$d\colon \AF\grd(V)\to\AF\grd(V)$.
We can not quite define a contracting homotopy in the same way,
but simply because we have not replaced $\AF_0(V)=\tensalg{S}{V}$ by $I$.

Indeed, the image of $d\colon \AF_1(V,I)\to\AF_0(V,I)$ is
$I\subs \tensalg{S}{V}$ and hence there is a dg morphism 
\begin{equation}
\label{eq:rho}
  \rho\colon \AF\grd(V,I)\to \tensalg{S}{V}/I,
\end{equation}
where the codomain here is just an algebra concentrated in degree 0.

\begin{theorem}
\label{thm:main}
The map $\rho$ in \eqref{eq:rho} is a quasi-isomorphism, that is, 
the dg algebra $\AF\grd(V,I)$ in \eqref{eq:FVI}
is a free resolution of the algebra $A=\tensalg{S}{V}/I$. 
\end{theorem}

\begin{proof}
Because $\rho$ is surjective, we just want to show that $\ker\rho$ is
contractible. 
For this, we observe that $\ker\rho$
is the dg ideal
\begin{equation}
  \AF'\grd(V,I)\subs\AF\grd(V,I)
\end{equation}
obtained by also replacing $\AF_0(V)$ by $I$.
Then, the contracting homotopy $h$ can be defined on 
$\AF'\grd(V,I)$ by the rule \eqref{eq:Fhom} as it was on $\AF\grd(V)$
and so we can use it to deduce that $\AF'\grd(V,I)$
is contractible, as required.
\end{proof}

\begin{remark}
\label{rem:main1}
In some cases, e.g. when $S$ is a separable $k$-algebra
(so that $\otimes$ is exact), we note that
$\AF\grd(V,I)$ is a sub-dg-algebra of $\AF\grd(V)$,
and so the equations $d^2=0$, for  $\AF\grd(V,I)$, and $dh+hd=\id$, for $\AF'\grd(V,I)$,
follow by restriction from $\AF\grd(V)$.
However, in general, they follow rather because the combinatorial structure of 
$\AF\grd(V,I)$ is identical to that of $\AF\grd(V)$.
\end{remark}

\begin{remark}
\label{rem:main3}
Just as $\AF\grd(V)=\cobar \KF\grd(V)$, as in Remark~\ref{rem:AFcoalg},
we have $\AF\grd(V,I)=\cobar \KF\grd(V,I)$,
where 
\begin{equation}
\label{eq:KVI}
  K\grd(V,I) = \bigoplus_{\beta\in\Brakc\str} (V,I)^\beta,
\end{equation}
with $(V,I)^\beta$ now in homological degree $d$ when $\beta$ has $d-1$ pairs of brackets.

In particular, when $I=V\gr2$, we have $K\grd(V,I)=\Bi\grd(V)$ from \eqref{eq:Kfree},
so that Theorem~\ref{thm:main} yields Theorem~\ref{thm:V}.
Further, for a homogeneous ideal $I\subs V\gr2$, we have $K\grd(V,I)=\Bi\grd(V,I)$ from \eqref{eq:BiVI},
so that Theorem~\ref{thm:main} yields Theorem~\ref{thm:homog}.
\end{remark}

\goodbreak\section{Homological application} \label{sec6}

Here we restict to the case when 
$S$ is a separable $k$-algebra over a field $k$
and  $V$ is an $S,S$-bimodule over $k$,
i.e. the $S,S$-action on $V$ factors through
$S \otimes_k S^{\operatorname{op}}$.

Note that $S$ is a left and right $A$-module, with $A$ acting trivially,
and we would expect (in good cases) that we could use an
$A_\infty$-coalgebra $K\grd$ with $\cobar K\grd \simeq A$
to compute the (positive) $\Tor$-groups of $S$ and hence, by duality,
its $\Ext$ groups. 
We observe that the case in hand, 
with $K\grd=K\grd(V,I)$ as in \eqref{eq:KVI},
is a good one.

\begin{proposition}
\label{pr:KFtor}
If $A=\tensalg{S}{V}/I$, then
we have an isomorphism
\[
  H(\KF\grd(V,I))\cong \Tor\grd^{A}(S,S)
\]
of graded coalgebras.
\end{proposition}
\begin{proof}
The $A_\infty$-coalgebra $\KF\grd(V)$ is cocomplete; e.g. the
filtration given by
\[
  \KF\grd(V)(i) = \bigoplus_{\beta\in\Brakc1, |\beta|\leq i} V^\beta
\]
is admissible (see Appendix for definition).
Since $S$ is a separable $k$-algebra, 
$\KF\grd(V,I)$ is a $A_\infty$-sub-coalgebra of $\KF\grd(V,I)$
and is therefore also cocomplete.
Since Theorem~\ref{thm:main} tells us that
$\AF\grd(V,I)=\cobar\KF\grd(V,I)$ is a dg resolution of $A$,
we deduce from Proposition~\ref{prop:Tor} that
$H(\KF\grd(V,I))\cong \Tor^{A}(S,S)$ as graded coalgebras.
\end{proof}

It follows that $\Hom_S(\KF\grd(V,I),S)$ 
is an $A_\infty$-algebra whose cohomology
is 
$\Ext^\bullet_{A}(S,S)$, in positive degrees.

\goodbreak\section{Appendix} \label{appendix}

Again we restrict to the case where $S$ is a separable $k$-algebra over a field $k$.
We extend part of the bar-cobar formalism for dg algebras and cocomplete dg coalgebras 
over $S$ to cocomplete $A_\infty$-coalgebras, following \cite{Lef}, $\cite{Ke1}$ and $\cite{Ke0}$. 
As a change of notation, we write $\Omega_\infty C$ for the cobar construction of an 
$A_\infty$-coalgebra $C$. 
This coincides with the more classical cobar construction $\Omega C$,
in the special case that $C$ is a dg coalgebra.
The bar construction of a dg algebra $A$ is denoted $BA$.

A \emph{filtration} on an $A_\infty$-coalgebra $C$ is a sequence $C(0)\subseteq C(1) \subseteq \ldots\subseteq C$ of graded sub-bimodules such that for all $n\geq 1$, the
coproduct $\Delta_n:C \to C^n$ is compatible with the (induced) filtrations on $C$ and $C^n$. A filtration is \emph{admissible} if $C(0)=0$ and $C= \cup C(i)$. We say $C$ is \emph{cocomplete} if it supports an admissible filtration; this is equivalent to the usual definition if $C$ is a dg coalgebra.
See \cite[\S 9.3, Remark]{Po} for the related notion of conilpotent (curved) $A_\infty$-coalgebra.

A morphism $f=(f_i):C'\to C$ between $A_\infty$-coalgebras is a \emph{quasi-isomorphism} 
if $f_1$ is a quasi-isomorphism of the underlying complexes of $C'$ and $C$, 
and a \emph{weak equivalence} if the morphism 
$\Omega_\infty f: \Omega_\infty C'\to\Omega_\infty C$
of cobar constructions is a quasi-isomorphism of dg algebras.
A morphism $f:C'\to C$ between cocomplete $A_\infty$-coalgebras is a
\emph{filtered quasi-isomorphism} if admissible filtrations can be chosen 
on $C'$ and on $C$ so that the maps $f_i:C' \to C^i$ are 
compatible with the (induced) filtrations on $C'$ and $C^i$, and the morphism
$\AG(f_1): \AG(C)\to \AG(C')$ induced by $f_1$ is a quasi-isomorphism. 
Note that a filtered quasi-isomorphism is a quasi-isomorphism.

For any dg algebra $A$, the classical bar-cobar resolution
gives a canonical morphism
$\varepsilon_A\colon \Omega_\infty B A \to A$ of dg algebras.
Taking $A=\Omega_\infty C$ for an $A_\infty$-coalgebra $C$, 
we obtain a corresponding map $\epsilon_C: B\Omega_\infty C \to C$
of $A_\infty$-coalgebras.
In fact, $\epsilon_C$ is a weak equivalence, 
because $\varepsilon_A$ is a quasi-isomorphism
(see, e.g. \cite[Lemme 1.3.2.3(b)]{Lef}),
but we can say more if $C$ is cocomplete.

\begin{lemma}
\label{filteredqis} 
\begin{enumerate}
\item If $A\to A'$ is a quasi-isomorphism of dg algebras, 
then the induced morphism $BA\to BA'$ is a filtered quasi-isomorphism.
\item If $C$ is a cocomplete $A_\infty$-coalgebra, 
then $\epsilon_C:B\Omega_\infty C \to C$ is a filtered quasi-isomorphism.
\end{enumerate}
\end{lemma}

\begin{proof}
The induced morphism $BA\to BA'$ is a
filtered quasi-isomorphism with respect to the
primitive filtrations of $BA$ and $BA'$.
For the second statement,
equip $B\Omega_\infty C$ with the admissible filtration induced by a given admissible filtration on $C$. Then $\epsilon_C:B\Omega_\infty C\to C$ is filtered, 
and it suffices to show that
$\AG_i((\epsilon_C)_1):
\AG_i(B\Omega_\infty C)\to \AG_i(C)$
is a quasi-isomorphism for all $i\geq 1$.
Put $W=C[1]$, so that 
$$B\Omega_\infty C=T^+\left(T^+(W)[-1]\right)$$
Since the filtration on $C$ is admissible, $\AG_i(C^j)=0$ if $j>i$.
Equip 
\[
  \AG_i\left(B\Omega_\infty C\right)
 =\AG_i\left(\bigoplus_{i_1+\ldots+i_k\leq i}
 W^{i_1}[-1]\otimes\ldots \otimes W^{i_k}[-1]\right)
\]
with the filtration
\[
  F_l=\AG_i\left(\bigoplus_{i+1-l\leq i_1+\ldots+i_k\leq i}
  W^{i_1}[-1]\otimes\ldots
 \otimes W^{i_k}[-1]\right),\quad l\geq 0.
 \]
 Then 
 $\AG_i((\epsilon_C)_1)$ is a surjective map of complexes with
 kernel $F_{i-1}$,
  and, for each $1 \leq l \leq i-1$, the subquotient complex
  \[
  F_l/F_{l-1} = 
  \bigoplus_{i_1+\ldots+i_k = i+1-l }
  W^{i_1}[-1]\otimes\ldots
 \otimes W^{i_k}[-1]
 \]
 is acyclic, with
a contracting homotopy vanishing on components with $i_1=1$ and given by isomorphisms
\[
W^{i_1}[-1]\otimes\ldots \otimes W^{i_k}[-1]
\rightarrow
W[-1]\otimes W^{i_1-1}[-1]
\otimes W^{i_2}[-1]
\otimes\ldots \otimes W^{i_k}[-1]
\]
otherwise.
\end{proof}

\begin{lemma} 
\label{cocompleteisos} 
Let $f:C'\to C$ be a morphism of cocomplete $A_\infty$-coalgebras.
\begin{enumerate}
\item If $f$ is a filtered quasi-isomorphism, then it is a weak equivalence.
\item If $f$ is a weak equivalence, then it is a quasi-isomorphism.
\end{enumerate}
\end{lemma}

\begin{proof}
 For the first part, the proof given by Lefevre \cite[Lemma 1.3.2.2]{Lef} 
for filtered quasi-isomorphisms of cocomplete dg coalgebras goes through without change.
 For the second, consider the commutative diagram
 $$
 \xymatrix{
 C' \ar[r]^f & C \\
 B\Omega_\infty C' \ar[u]^{\epsilon_{C'}} \ar[r]^{B\Omega_\infty f} & B\Omega_\infty C \ar[u]_{\epsilon_{C}}
 }$$
By Lemma~\ref{filteredqis}, $\epsilon_C$, $\epsilon_{C'}$
and $B\Omega_\infty f$ are quasi-isomorphisms. 
Hence $f$ is a quasi-isomorphism.
\end{proof}

\begin{corollary} [Keller \cite{Ke0}]
\label{cor:qi=we}
Let $f\colon C'\to C$ be a morphism of $A_\infty$-coalgebras 
which is also compatible with strictly positive gradings on $C'$ and $C$. 
If $f$ is a quasi-isomorphism, then it is a weak equivalence.
\end{corollary}

\begin{proof}
The additional positive gradings give rise to admissible filtrations 
on $C'$ and $C$ in an obvious way. 
If $f$ is a quasi-isomorphism, then it is automatically a filtered 
quasi-isomorphism, and then Lemma~\ref{cocompleteisos} applies.
\end{proof}

\begin{proposition}\label{prop:Tor}
Suppose that $C$ is a cocomplete $A_\infty$-algebra and
we have a quasi-isomorphism $\Omega_\infty C\to A$ of dg algebras. 
Then $C$ is weakly equivalent to $BA$. 
In particular we have an isomorphism of graded coalgebras
\[
  H(C)\cong \Tor^{A}(S,S).
\]
\end{proposition}

\begin{proof}
By Lemma~\ref{filteredqis}, we have filtered quasi-isomorphisms
$B\Omega_\infty C\to BA$ and
$B\Omega_\infty C \to C$.
In particular $C$ and $BA$ are weakly equivalent, and thus quasi-isomorphic, 
by Lemma~\ref{cocompleteisos}. 
Finally, recall that $H(BA)\cong \Tor^{A}(S,S)$.
\end{proof}
 
\goodbreak

Coordinates:

JC: City University, London; Joseph.Chuang.1@city.ac.uk

AK: University of Bath; a.d.king@bath.ac.uk


\begin{thebibliography}{99}

\bibitem{BGS} A.~Beilinson, V.~Ginzburg, W.~Soergel,
  \textit{Koszul duality patterns in representation theory}
  J.A.M.S. \textbf{9} (1996) 473--527.

\bibitem{Hi} V.~Hinich.
\textit{Homological algebra of homotopy algebras},
Comm. Algebra \textbf{25} (1997), no. 10, 3291--3323.

\bibitem{Ke1} B.~Keller,
\textit{Koszul duality and coderived categories
(after K. Lef\'evre)}, notes of a 50-minute talk given at the Frobenius conference in Toru-n (Poland) in September 2003.

\bibitem{Ke0} B.~Keller,
\textit{Notes on minimal models}, available at the author's homepage
\url{http://www.math.jussieu.fr/~keller}.



\bibitem{Lef} K.~Lef\`evre,
\textit{Sur le $A_\infty$-cat\'egories},
Th\`ese, Universit\'e Paris 7, 2003.
Available at
\url{http://www.math.jussieu.fr/~lefevre/publ.html}.

\bibitem{Po} L.~Positselski. 
\textit{Two kinds of derived categories, Koszul duality
and comodule-contramodule correspondence}, 
Mem. Amer. Math. Soc. \textbf{212} (2011), no. 996, vi+133 pp. 

\bibitem{St63} J.D.~Stasheff,
\textit{Homotopy associativity of H-spaces. I, II.} 
Trans. A.M.S. \textbf{108} (1963), 275--292; ibid. 293--312.

\end{thebibliography}
\end{document}